\documentclass{amsart}
\usepackage{eurosym}
\usepackage{amsmath}
\usepackage{amsfonts}
\usepackage{hyperref}
\usepackage{graphicx}
\usepackage{cite}
\makeatletter
\renewcommand{\author}[2][]{%
	\def\@tempa{#1}
	\ifx\@empty\authors
	\ifx\@tempa\@empty
	\gdef\shortauthors{#2}%
	\else
	\gdef\shortauthors{#1}%
	\fi
	\gdef\authors{\author{#2}}%
	\else
	\ifx\@tempa\@empty
	\g@addto@macro\shortauthors{\and#2}%
	\else
	\g@addto@macro\shortauthors{\and#1}%
	\fi
	\g@addto@macro\authors{\and\author{#2}}%
	\fi
}
\renewcommand{\address}[2][]{\g@addto@macro\authors{\address{#1}{#2}}}
\def\@setauthors{%
	\begin{center}%
		\footnotesize
		\vspace{20pt}
		\let\and\@empty
		\def\author##1{\advance\@tempcnta\@ne}%
		\def\address##1##2{\advance\@tempcntb\@ne}%
		\@tempcnta=\z@  \@tempcntb=\z@
		\authors
		\ifnum\@tempcnta>\@ne \ifnum\@tempcntb=\@ne
		\oneaddress
		\else
		\sepaddresses
		\fi
		\else
		\oneaddress
		\fi
	\end{center}%
}
\def\oneaddress{%
	\begingroup
	\let\author\@iden \let\address\@gobbletwo
	\renewcommand{\andify}{%
		\nxandlist{\unskip, }{\unskip{} and~}{\unskip, and~}}%
	\uppercasenonmath\authors
	\andify\authors
	\authors
	\endgroup
	\begingroup \let\and\relax \let\author\@gobble
	\def\address##1##2{\unskip\\[10pt] \itshape##2}%
	\authors
	\endgroup
}
\def\sepaddresses{%
	\begingroup
	\baselineskip10\p@\relax
	\def\address##1##2{ ({\itshape##2}\/)}
	\def\author##1{\def\temp{##1}\leavevmode\uppercasenonmath\temp\temp}%
	\nxandlist
	{,\\[\baselineskip]}
	{\\[\baselineskip] \textsc{\lowercase{and}}\\[\baselineskip]}
	{,\\[\baselineskip]\textsc{\lowercase{and}}\\[\baselineskip]}
	\authors 
	\authors
	\endgroup
}
\def\maketitle{\par
	\@topnum\z@
	\@setcopyright
	\thispagestyle{firstpage}%
	\uppercasenonmath\shorttitle
	\ifx\@empty\shortauthors \let\shortauthors\shorttitle
	\else
	\newcommand{\@xuppercasenonmath}[1]{\toks@\@emptytoks
		\@xp\@skipmath\@xp\@empty##1$$%
		\edef##1{\@nx\protect\@nx\@upprep\the\toks@}}%
	\@xuppercasenonmath\shortauthors
	\def\@@and{AND}
	\renewcommand{\andify}{%
		\nxandlist{\unskip, }{\unskip{ }\@@and{ }}{\unskip, \@@and{ }}}%
	\andify\shortauthors
	\fi
	\@maketitle@hook
	\begingroup
	\@maketitle
	\endgroup
	\c@footnote\z@
	\@cleartopmattertags
}
\def\@maketitle{%
	\normalfont\normalsize
	\let\@makefntext\noindent
	\@adminfootnotes
	\ifx\@empty\addresses\else \@footnotetext{\@setotheraddresses}\fi
	\global\topskip68\p@\relax
	\@settitle
	\ifx\@empty\authors \else \@setauthors \fi
	\ifx\@empty\@dedicatory
	\else
	\baselineskip26\p@
	\vtop{\centering{\footnotesize\itshape\@dedicatory\@@par}%
		\global\dimen@i\prevdepth}\prevdepth\dimen@i
	\fi
	\toks@\@xp{\shortauthors}\@temptokena\@xp{\shorttitle}%
	\edef\@tempa{\@nx\markboth{\the\toks@}{\the\@temptokena}}\@tempa
	\@setabstract
	\normalsize
	\if@titlepage
	\newpage
	\else
	\dimen@34\p@ \advance\dimen@-\baselineskip
	\vskip\dimen@\relax
	\fi
} 
\renewcommand{\thanks}[1]{%
	\ifx\@empty\thankses
	\gdef\thankses{\thanks{#1}}%
	\else
	\g@addto@macro\thankses{\endgraf\thanks{#1}}%
	\fi}
\def\@setthanks{\def\thanks##1{\noindent##1\@addpunct.}\thankses}
\renewcommand{\curraddr}[2][]{%
	\ifx\@empty\addresses
	\gdef\addresses{\curraddr{#1}{#2}}%
	\else
	\g@addto@macro\addresses{\endgraf\curraddr{#1}{#2}}%
	\fi}
\renewcommand{\email}[2][]{%
	\ifx\@empty\addresses
	\gdef\addresses{\email{#1}{#2}}%
	\else
	\g@addto@macro\addresses{\endgraf\email{#1}{#2}}%
	\fi}
\renewcommand{\urladdr}[2][]{%
	\ifx\@empty\addresses
	\gdef\addresses{\urladdr{#1}{#2}}%
	\else
	\g@addto@macro\addresses{\endgraf\urladdr{#1}{#2}}%
	\fi}
\def\@setotheraddresses{%
	\def\curraddr##1##2{\noindent
		\emph{Current address\@ifnotempty{##1}{ of ##1}}:\space
		##2\@addpunct.}%
	\def\email##1##2{\noindent
		\emph{E-mail address\@ifnotempty{##1}{ of ##1}}:\space
		\texttt{##2}}%
	\def\urladdr##1##2{\noindent
		\emph{WWW address\@ifnotempty{##1}{ of ##1}}:\space
		\texttt{##2}}%
	\addresses
}
\let\enddoc@text\relax
\makeatother
\newtheorem{theorem}{Theorem}
\theoremstyle{plain}

\newtheorem{definition}{Definition}

\newtheorem{lemma}{Lemma}

\newtheorem{remark}{Remark}

\numberwithin{equation}{section}

\begin{document}
\title[Exponential-type Samling Operator]{Approximation Properties of Mellin-Steklov Type Exponential Sampling Series}
\author{Dilek Ozer}
\address{Selcuk University, Faculty of Science, Department of Mathematics, Selcuklu, 42003, Konya, Türkiye}
\email{dilekozer@yahoo.com}
\author{Sadettin Kursun}
\address{National Defence University, Turkish Military Academy, Department of Basic Sciences, Cankaya, 06420,
	Ankara, Türkiye}
\email{sadettinkursun@yahoo.com}
\author{Tuncer Acar}
\address{Selcuk University, Faculty of Science, Department of Mathematics, Selcuklu, 42003, Konya, Türkiye}
\email{tunceracar@ymail.com}
\maketitle

\begin{abstract}
In this paper, we introduce Mellin-Steklov exponential sampling operators of order $r,r\in\mathbb{N}$, by considering appropriate Mellin-Steklov integrals. We investigate the approximation properties of these operators in continuous bounded spaces and $L^p,1\leq p<\infty$ spaces on $\mathbb{R}_+$. By using the suitable modulus of smoothness, it is given high order of approximation. Further, we present a quantitative Voronovskaja type theorem and we study the convergence results of newly constructed operators in logarithmic weighted spaces of functions. Finally, the paper provides some examples of kernels that support the our results.  
\end{abstract}

\section{Introduction}
 The classical sampling type operators attributed to Whittaker-Kotel'nikov-Shannon (see, \cite{Whittaker,Kotel'nikov,Shannon}) are defined by 
 \begin{equation}
 	\left(G_w f\right)\left(t\right):=\sum_{k\in\mathbb{Z}}f\left(\frac{k}{w}\right)\operatorname{sinc}\left(wt-k\right), \ w>0, t\in\mathbb{R}, \label{CSS}
 \end{equation}
where $\operatorname{sinc}$ function is given by $\operatorname{sinc}\left(t\right)=\frac{\sin\left(\pi t\right)}{\pi t},t\in\mathbb{R}\backslash\left\{0\right\}$ and $\operatorname{sinc}\left(0\right)=1$. We know that 
\begin{equation*}
	\sum_{k\in\mathbb{Z}}f\left(\frac{k}{w}\right)\operatorname{sinc}\left(wt-k\right)=f\left(t\right),
\end{equation*}
where $f:\mathbb{R}\rightarrow\mathbb{R}$ is a simultaneously band-limited signal. 

Later, P. L. Butzer et al. \cite{butzeretal} introduced a generalization of the operators \eqref{CSS} by using the kernel function $\varphi:\mathbb{R}\rightarrow\mathbb{R}$ which satisfies the certain assumptions instead of $\operatorname{sinc}$ function. The operators are in the form
\begin{equation}
	\left(G_w^\varphi\right)\left(t\right):=\sum_{k\in\mathbb{Z}}f\left(\frac{k}{w}\right)\varphi\left(wt-k\right), \ w>0,t\in\mathbb{R} \label{GSS}
\end{equation}
for any function $f:\mathbb{R}\rightarrow\mathbb{R}$ that makes the above series absolutely convergent on $\mathbb{R}$. In the same paper, the authors investigated approximation properties of the operators \eqref{GSS} in space of continuous functions on $\mathbb{R}$. 

It is known that the operators $G_w^\varphi$ are not suitable enough to approximate integrable functions $f:\mathbb{R}\rightarrow\mathbb{R}$ which are not continuous and we know that Kantorovich's idea allows for the approximation of not necessarily continuous functions. For this reason, Bardaro et al. \cite{bardaroetal1} introduced a Kantorovich variant of the family of operators \eqref{GSS}. The Kantorovich-type generalized sampling series is defined by 
\begin{equation}
	\left(K_w^\varphi f\right)\left(t\right):=\sum_{k\in\mathbb{Z}}\varphi\left(wt-k\right)\left[w\int_{\frac{k}{w}}^{\frac{k+1}{w}}f\left(y\right)dy\right], w>0, t\in\mathbb{R} \label{KTGSS}
\end{equation}
for locally integrable function defined on $\mathbb{R}$. These operators present an approximation method for functions which belong to $L^1$ spaces on $\mathbb{R}$. 

In order to present an approximation method for functions belonging to $L^p,1\leq p<\infty$ spaces, Bardaro and Mantellini \cite{bardaromantellini2} introduced Durrmeyer modification of \eqref{GSS}. Using a general convolution integral instead of integral means on the interval $\left[\frac{k}{w},\frac{k+1}{w}\right]$, the Durrmeyer-type generalized sampling series is given by 
\begin{equation}
	\left(G_w^{\varphi,\psi}f\right)\left(t\right):=\sum_{k\in\mathbb{Z}}\varphi\left(wt-k\right)\left[w\int_{\mathbb{R}}\psi\left(wu-k\right)f\left(u\right)du\right], w>0, t\in\mathbb{R}, \label{DTGSS}
\end{equation} 
where $\psi$ is a kernel function satisfying the certain assumptions such as $\varphi$.

 For other publications in the literature on this subject, see also \cite{bardaromantellini1,costarellietal,costarellisambucini,costarellivinti,cosvinti,bardaroetal,costarellietal2,cospiconivinti,acarbor,acarbor2,borislav,turgacar,alagoz}. Furthermore, for the approximation properties of generalized sampling operators and their different forms in weighted spaces of continuous functions (see, \cite{aeaws,aoactv,acarturgay, atap,ozerturgayacar}). 
 
 In 1980s,  a group which consists of physicists and engineers established the exponential form of the operatos \eqref{CSS} to solve the certain problems in optical physics, like light-scattering, diffraction, radio-astronomy and so on (see, \cite{berteropike,casasent,gori}). The classical exponential form of sampling series is given by
 \begin{equation}
 	\left(E_{c,w}f\right)\left(x\right):=\sum_{k\in\mathbb{Z}}f\left(e^{\frac{k}{w}}\right)\operatorname{lin}_{\frac{c}{w}}\left(e^{-k}x^w\right),\ c\in\mathbb{R}, w>0,x\in\mathbb{R}_+, \label{EFCSS}
 \end{equation}
where $\operatorname{lin}_c$ function is defined by $\operatorname{lin}_c\left(x\right)=\frac{t^{-c}}{2\pi i}\frac{x^{\pi}-x^{-\pi i}}{\log x},x\in\mathbb{R}_+\backslash\left\{1\right\}$ and $\operatorname{lin}_c\left(1\right)=1$. If $f$ is a Mellin band-limited signal, then we have $\left(E_{c,w}f\right)\left(x\right)=f\left(x\right)$ for every $x\in\mathbb{R}_+$ (see, \cite{butzerc}). 

For functions not necessarily Mellin band-limited, Bardaro et al. \cite{bardaroc} established a generalization of the exponential sampling series \eqref{EFCSS} by using the kernel function $\chi:\mathbb{R}_+\rightarrow\mathbb{R}$ satisfying the certain assumptions instead of $\operatorname{lin}_c$ function. The generalized exponential sampling series is defined by 
\begin{equation}
	\left(E_w^\chi f\right)\left(x\right):=\sum_{k\in\mathbb{Z}}f\left(e^{\frac{k}{w}}\right)\chi\left(e^{-k}x^w\right), \ w>0,x\in\mathbb{R}_+, \label{GESS}
\end{equation}
where $f:\mathbb{R}_+\rightarrow\mathbb{R}$ is a signal that makes the above series absolutely convergent on $\mathbb{R}_+$. In the same paper, the authors presented convergence results of the operators \eqref{GESS} in space of log-uniformly continuous functions on $\mathbb{R}_+$. 

Similar to the construction of the operators \eqref{KTGSS}, Angamuthu and Bajpeyi \cite{angamuthubajpeyi} introduced a Kantorovich form of exponential sampling operators. The Kantorovich-type exponential sampling operators are in the form 
\begin{equation}
	\left(I_w^{\chi}f\right)\left(x\right):=\sum_{k\in\mathbb{Z}}\chi\left(e^{-k}x^w\right)\left[w\int_{\frac{k}{w}}^{\frac{k+1}{w}}f\left(e^y\right)dy\right], w>0,x\in\mathbb{R}_+, \label{KTGESS}
\end{equation}
where $f$ is a locally integrable signal defined on $\mathbb{R}_+$. These operators an approximation method for functions belonging to $L^1$ spaces on $\mathbb{R}_+$. 

In order to present an approximation method for functions belonging to Mellin-Lebesgue ($X_c^p,c\in\mathbb{R},1\leq p<\infty$) spaces, Bardaro and Mantellini \cite{bardaromantellini} constructed a Durrmeyer modification of the operators \eqref{KTGESS} using a certain general convolution integral. The Durrmeyer-type exponential sampling operators are defined by 
\begin{equation}
	\left(I_w^{\chi,\Psi}f\right)\left(x\right):=\sum_{k\in\mathbb{Z}}\chi\left(e^{-k}x^w\right)\left[w\int_{\mathbb{R}_+}\Psi\left(e^{-k}u\right)f\left(u\right)\frac{du}{u}\right], w>0,x\in\mathbb{R}_+
\end{equation}
where $\Psi$ is a kernel function which satisfies the suitable assumptions such as $\chi$.

For other publications on the exponential sampling series and its different forms (see, \cite{bajpeyiangamuthu,balsamo,aak,bajpeyi,kursunaralacarr,kursunaralacar2,kursunacarr}). Moreover, for the approximation properties of generalized exponential sampling series and its different forms in logarithmic weighted spaces of continuous functions (see also \cite{acarkursun,acarekekursun,acarkursunacar}).

Very recently, Costarelli \cite{costarelli2024} has introduced a new form of the operators \eqref{GSS} named Steklov sampling operators of order $r,r\in\mathbb{N}$. In doing so, the author has considered the following Steklov-type integrals: 
\begin{equation}
	f_{r,h}\left(t\right):=\left(-h\right)^{-r}\int_{0}^{h}\cdot\int_{0}^{h}\sum_{m=1}^{r}\left(-1\right)^{r-m+1}{r\choose m}f\left(t+\frac{m}{r}\left(u_1+\dots+u_r\right)\right)du_1\dots du_r, \label{STI}
\end{equation}  
where $f:\mathbb{R}\rightarrow\mathbb{R}$ is a locally integrable function with $r\in\mathbb{N}$ and $h>0$. Considering the integrals \eqref{STI}, the Steklov sampling operators of order $r$ are defined by 
\begin{align}
	&\left(S_{w,r}^\varphi f\right)\left(t\right):=\sum_{k\in\mathbb{Z}}\varphi\left(wt-k\right) \label{SSO_r}\\
	&\notag \times\left[w^r\int_{0}^{\frac{1}{w}}\dots\int_{0}^{\frac{1}{w}}\sum_{m=1}^{r}\left(-1\right)^{1-m}{r\choose m}f\left(\frac{k}{w}+\frac{m}{r}\left(u_1+\dots+u_r\right)\right)du_1\dots du_r\right], \ w>0, t\in\mathbb{R} 
\end{align} 
for any locally integrable function $f:\mathbb{R}\rightarrow\mathbb{R}$ for
which the above series are convergent.

Our aim in this paper is to construct Mellin-Steklov exponential sampling operators of order $r,r\in\mathbb{N}$.  To do this, we first introduce the appropriate Mellin-Steklov integrals and mention their necessary properties, which are be used in the paper. Secondly, we present the convergence results of the newly constructed operators both in continuous function spaces and in $L^p,1\leq p<\infty$ spaces on $\mathbb{R}_+$. Moreover, we obtain the high order of approximation for these operators via suitable logarithmic modulus of continuity. Further, we give a quantitative Voronovskaja type theorem. Finally, we investigate the approximation properties of the present operators in logarithmic weighted spaces of functions.  
\section{Basic Notations and Auxiliary Results}

By $\mathbb{N}$ and $\mathbb{Z}$, let us denote the set of positive integers and integers, respectively. Furthermore, by $\mathbb{R}$ and $\mathbb{R}_{+}$, we denote the sets of all real and positive real numbers respectively.

Let $C\left( \mathbb{R}_{+}\right)$ be the space of all continuous functions defined on $\mathbb{R}_+$ and $CB\left( \mathbb{R}^{+}\right) $ is the
space of all bounded functions $f$ that belonging to $C\left(\mathbb{	R}_{+}\right) $. Let $\left\| \cdot\right\|_\infty $ stand for the sup-norm in $CB\left(\mathbb{R}_+\right)$. We say that a function $f\in C\left( \mathbb{R}
_{+}\right) $ is log -uniformly continuous on $\mathbb{R}_{+}$ if for
any $\varepsilon >0$ there exists  $\delta_\varepsilon >0$ such that $\left\vert
f\left( u\right) -f\left( v\right) \right\vert <\varepsilon $ whenever $
\left\vert \log u-\log v\right\vert \leq\delta $ for any $u,v\in \mathbb{R}
_{+} $. We know that in general a log-uniformly continuous function is not necessarily (usual) uniformly continuous function and conversely, but these notions are equivalent on every compact intervals of $\mathbb{R}_+$ (see, \cite{bardaroc}).  By $\mathcal{C}\left(\mathbb{R}_+\right)$, we denote the subspace of $CB\left(\mathbb{R}_+\right)$ consisting of all log-uniformly continuous functions. Also, we denote by $CB_{\text{comp}}\left( \mathbb{R}_{+}\right) $ the
subspace of $CB\left( \mathbb{R}_{+}\right) $ comprising all functions  with  compact supports in $\mathbb{R}_+$. 

Finally, we shall denote by $L^p\left(\mathbb{R}_+\right),1\leq p<\infty$ the usual Lebesgue spaces which contains all Lebesgue measurable functions such that 
\begin{equation*}
	\left\|f \right\|_p:=\left(\int_{\mathbb{R}_+}\left|f\left(x\right)\right|^p\frac{dx}{x}\right)^{\frac{1}{p}}.
\end{equation*}

Throughout this paper, a continuous function $\chi :\mathbb{R}_{+}\longrightarrow \mathbb{R}$
is called a kernel if the following assumptions hold:
\begin{itemize}
\item[$\left(\chi_1\right)$] $\chi$ in $L^1\left(\mathbb{R}_+\right)$ and it is bounded on $\left[\frac{1}{e},e\right]$;
\item[$\left(\chi_2\right)$] we have 
\begin{equation*}
	\sum_{k\in\mathbb{Z}}\chi\left(e^{-k}u\right)=1
\end{equation*}
for every $u\in\mathbb{R}_+$ and 
\begin{equation*}
	M_{0}\left( \chi \right) :=\sup_{u\in\mathbb{R}_+}\sum_{k\in\mathbb{Z}}\left\vert \chi \left( e^{-k}u\right) \right\vert
	<+\infty
\end{equation*}
\item[$\left(\chi_3\right)$] there exists $\gamma >0$ such that
\begin{equation*}
	\lim_{\gamma\rightarrow\infty}\sum_{\left|k-\log u\right|>\gamma}\left|\chi\left(e^{-k}u\right)\right|=0
\end{equation*}
uniformly with respect to $u\in\mathbb{R}_+$. 
\end{itemize}
The class of functions satisfying the assumptions $\left( \chi _1\right), \left( \chi _2\right) $ and $\left(\chi_3\right)$ will be denoted by $\phi .$

For $i\in \mathbb{N}$, the algebraic moments of order $i$ of a kernel $\chi
\in \phi $ are defined by
\begin{equation*}
	m_{i}\left( \chi ,u\right) :=\underset{k\in \mathbb{Z}}{\sum }\chi \left(
	e^{-k}u\right) \left( k-\log u\right) ^{i},\text{ \ }u\in \mathbb{R}^{+}.
\end{equation*}

and the absolute moments of order $\alpha>0$ of a kernel $\chi\in\phi$ are given by
\begin{equation*}
	M_{\alpha }\left( \chi \right) :=\sup_{u\in\mathbb{R}_+}\sum_{k\in\mathbb{Z}}%
	\left\vert \chi \left( e^{-k}u\right) \right\vert \left\vert k-\log
	u\right\vert ^{\alpha }.
\end{equation*}
\begin{remark}
	It is shown in \cite{barmansch} that if $M_{\alpha }\left( \chi \right) <\infty $ we also have $M_{\beta }\left( \chi \right) <\infty $ for $0<\beta <\alpha .$
\end{remark}
We need the following Lemma when giving the convergence result in $L^p$. 
\begin{lemma}(\cite{barmantit})\label{Lemma1}
	For every $\gamma
	>0$ and $\varepsilon>0$,  there is a constant $M>0$ such that
	\begin{equation*}
		\int_{x\not\in\left[e^{-M},e^M\right]}w\left\vert \chi
		\left( e^{-k}x^{w}\right) \right\vert \frac{dx}{x}<\varepsilon
	\end{equation*}
	for sufficiently large $w$ and $k\in \mathbb{Z}$ such that $\frac{k}{w}\in %
	\left[ -\gamma ,\gamma \right]$.
\end{lemma}
Now, we define the logarithmic modulus of smoothness of order $m\in\mathbb{N}$ for $f\in CB\left(\mathbb{R}_+\right)$ and $\delta>0$ as follows: 
\begin{equation}
	\omega_m\left(f,\delta\right):=\sup_{0<\log h\leq \delta}\left\|\Delta_h^m f \right\|, \label{LMOS}
\end{equation}
where $\Delta_hf\left(x\right):=f\left(xh\right)-f\left(x\right),\ x,h\in\mathbb{R}_+$ and $\Delta_h^m:=\Delta_h\left(\Delta_h^{m-1}\right)$. The expanded form of $\Delta_h^m f$ is in the form
\begin{equation*}
	\Delta_h^mf\left(x\right):=\sum_{j=0}^{m}\left(-1\right)^{m-j}{m \choose j}f\left(xh^j\right),\ x\in\mathbb{R}_+.
\end{equation*} 
The logarithmic modulus of smoothness given in \eqref{LMOS} has the following properties.
\begin{lemma}
Let $\delta>0$. Then 
\begin{itemize}
\item[a)]for $f\in \mathcal{C}\left(\mathbb{R}_+\right)$, $\displaystyle\lim_{\delta\rightarrow0}\omega_m\left(f,\delta\right)=0$,
\item[b)] for any $\lambda>0$, $\omega_m\left(f,\lambda\delta\right)\leq\left(\lambda+1\right)^m\omega_m\left(f,\delta\right)$
\end{itemize}
hold.
\begin{proof}
	The proof can be given similarly to the classical modulus of smoothness given in \cite{devorelorentz}. 
\end{proof}
\end{lemma}
In order to construct Mellin-Steklov type exponential sampling operators, we finally introduce Mellin-Steklov integrals. The Mellin-Steklov integrals of order $r\in\mathbb{N}$ with $h>1$ are defined by 
\begin{equation}
		F_{r,h}\left( x\right) :=\left( -\log h\right) ^{-r}\underset{1}{\overset{h}{%
			\int }}...\underset{1}{\overset{h}{\int }}\sum_{m=1}^{r}\left( -1\right) ^{r-m+1}\binom{r}{m}f\left( x\left(
	t_{1}...t_{r}\right) ^{\frac{m}{r}}\right) \frac{dt_{1}}{t_{1}}...\frac{
		dt_{r}}{t_{r}}, \ x\in\mathbb{R}_+ \label{MSI}
\end{equation}
for any locally integrable function $f:\mathbb{R}
_{+}\longrightarrow \mathbb{R}$. If $h=e^{\frac{1}{w}}$, we can write that
\begin{equation*}
	F_{r,e^{\frac{1}{w}}}\left( e^{\frac{k}{w}}\right) :=w^{r}\underset{1}{%
		\overset{e^{\frac{1}{w}}}{\int }}...\underset{1}{\overset{e^{\frac{1}{w}}}{%
			\int }}\sum_{m=1}^{r}\left( -1\right) ^{1-m}%
	\binom{r}{m}f\left( e^{\frac{k}{w}}\left( t_{1}...t_{r}\right) ^{\frac{m}{r}%
	}\right) \frac{dt_{1}}{t_{1}}...\frac{dt_{r}}{t_{r}}.
\end{equation*}
\begin{remark}
	For any locally integrable function $f:\mathbb{R}_
	{+}\longrightarrow \mathbb{R}$, with $r\in \mathbb{N}$ and $h>1$, we have
	\begin{eqnarray*}
		F_{r,h}\left( x\right) -f\left( x\right) &=&\frac{1}{\left( -1\right)
			^{r}\left( \log h\right) ^{r}}\underset{1}{\overset{h}{\int }}...\underset{1}
		{\overset{h}{\int }}\sum_{m=1}^{r}\left(
		-1\right) ^{r-m+1}\binom{r}{m}f\left( x\left( t_{1}...t_{r}\right) ^{\frac{m}{r}}\right) \frac{dt_{1}}{t_{1}}...\frac{dt_{r}}{t_{r}} \\
		&-&\frac{1}{\left( -1\right) ^{r}\left( \log h\right) ^{r}}\underset{1}{%
			\overset{h}{\int }}...\underset{1}{\overset{h}{\int }}\left( -1\right)
		^{r}f\left( x\right) \frac{dt_{1}}{t_{1}}...\frac{dt_{r}}{t_{r}} \\
		 &=&\frac{-1}{\left( -1\right) ^{r}\left( \log h\right) ^{r}}\underset{1}{%
			\overset{h}{\int }}...\underset{1}{\overset{h}{\int }}\sum_{m=0}^{r}\left( -1\right) ^{r-m}\binom{r}{m}f\left( x\left(
		t_{1}...t_{r}\right) ^{\frac{m}{r}}\right) \frac{dt_{1}}{t_{1}}...\frac{%
			dt_{r}}{t_{r}} \\
		&=&\frac{-1}{\left( -1\right) ^{r}\left( \log h\right) ^{r}}\underset{1}{%
			\overset{h}{\int }}...\underset{1}{\overset{h}{\int }}\Delta _{\left(
			t_{1}...t_{r}\right) ^{\frac{1}{r}}}^{r}f\left( x\right) \frac{dt_{1}}{t_{1}}
		...\frac{dt_{r}}{t_{r}}.
	\end{eqnarray*}
\end{remark}
\section{Construction of the operators and their Convergence results}
In this section, we first introduce the Mellin-Steklov exponential sampling operators by using the integrals given in \eqref{MSI}. Later, we give pointwise and uniform convergence results of these operators. Finally, we present $L^p,1\leq p<\infty$ convergence result for these operators.  

We are able to introduce the following. 
\begin{definition}
	Let $r\in \mathbb{N}$ be fixed. The Mellin-Steklov exponential sampling
	operators of order $r$ are defined by 
	\begin{align}
		\left( E_{w,r}^{\chi }f\right) \left( x\right) :=&\sum_{k\in\mathbb{Z}}F_{r,e^{\frac{1}{w}}}\left( e^{\frac{k}{w}
		}\right) \chi \left( e^{-k}x^{w}\right) \label{MSESO} \\
		=&\sum_{k\in\mathbb{Z}}\chi \left(
		e^{-k}x^{w}\right)\left\{ w^{r}\underset{1}{
			\overset{e^{\frac{1}{w}}}{\int }}...\underset{1}{\overset{e^{\frac{1}{w}}}{
				\int }}\sum_{m=1}^{r}\left( -1\right) ^{1-m}
		\binom{r}{m}f\left( e^{\frac{k}{w}}\left( t_{1}...t_{r}\right) ^{\frac{m}{r}
		}\right) \frac{dt_{1}}{t_{1}}...\frac{dt_{r}}{t_{r}}\right\}, \ x \in\mathbb{R}_+,w>0 \notag
	\end{align}
	for any locally integrable function $f:\mathbb{R}_{+}\rightarrow 
	\mathbb{R}$ for which the above series are convergent.
\end{definition}
\begin{remark}
	In \eqref{MSESO}, if we consider the situation $r=1$, we get the classical
	Kantorovich forms of exponential sampling series given in \eqref{KTGESS}.
\end{remark}
The above sampling-type series are well-defined for every $r\in\mathbb{N}$ and $w>0$, assuming, for example, that the functions $f:\mathbb{R}_+\rightarrow\mathbb{R}$ are bounded. We can easily that
	\begin{align}
		\left\vert \left( E_{w,r}^{\chi }f\right) \left( x\right) \right\vert
		&=\left\vert \sum_{k\in\mathbb{Z}}\chi \left( e^{-k}x^{w}\right)\left\{ w^{r}
		\underset{1}{\overset{e^{\frac{1}{w}}}{\int }}...\underset{1}{\overset{e^{
					\frac{1}{w}}}{\int }}\sum_{m=1}^{r}\left(
		-1\right) ^{1-m}\binom{r}{m}f\left( e^{\frac{k}{w}}\left(
		t_{1}...t_{r}\right) ^{\frac{m}{r}}\right) \frac{dt_{1}}{t_{1}}...\frac{
			dt_{r}}{t_{r}}\right\} \right\vert \notag\\
		&\leq\left\Vert f\right\Vert _{\infty }\left( 2^{r}-1\right) \sum_{k\in\mathbb{Z}}\left\vert \chi \left(
		e^{-k}x^{w}\right) \right\vert\left\{ w^{r}\underset{1}{\overset{e^{\frac{1}{w
			}}}{\int }}...\underset{1}{\overset{e^{\frac{1}{w}}}{\int }}\frac{dt_{1}}{
			t_{1}}...\frac{dt_{r}}{t_{r}}\right\} \notag \\
		&\leq
		\left\Vert f\right\Vert _{\infty }\left( 2^{r}-1\right) M_{0}\left( \chi
		\right). \label{ineq3.2}
	\end{align}
Now, we present the pointwise and uniform convergence results. 
\begin{theorem}\label{theorem1}
	Let $\chi\in\phi$ be a kernel. Then for every bounded function $f:\mathbb{R}_+\rightarrow\mathbb{R}$,
	\begin{equation*}
		\lim_{w\rightarrow\infty}\left( E_{w,r}^{\chi}f\right)
		\left( x\right) =f\left( x\right)
	\end{equation*}
	holds at any point of continuity $x\in \mathbb{R}_{+}$ of the function $f.$ Furthermore, if 
	$ \mathcal{C}\left( \mathbb{R}_{+}\right) $ we have
	\begin{equation*}
		\lim_{w\rightarrow\infty}\left\Vert E_{w,r}^{\chi}f-f\right\Vert _{\infty }=0.
	\end{equation*}
\end{theorem}
\begin{proof}
Let us start the first part of the theorem. Using the continuity of the function $f$ at the point $x\in\mathbb{R}_+$, we know that  for every $\varepsilon
	>0$ there exists $\delta >0$ such that $\left\vert f\left( x\right)
	-f\left( y\right) \right\vert <\varepsilon $ for any $y\in \mathbb{R}_{+}$
	for which $\left\vert \log x-\log y\right\vert \leq \delta $. Now, using the
	condition $\left( \chi _{1}\right)$, we can write what follows:
	\begin{align*}
		\left( E_{w}^{\chi ,r}f\right) \left( x\right) -f\left( x\right) &=\sum_{k\in\mathbb{Z}}F_{r,e^{\frac{1}{w}}}\left( e^{\frac{k}{w}}\right)
		\chi \left( e^{-k}x^{w}\right) -f\left( x\right) \sum_{k\in\mathbb{Z}}\chi \left( e^{-k}x^{w}\right) \\
		&=\sum_{k\in\mathbb{Z}}\chi
		\left( e^{-k}x^{w}\right)\left\{ w^{r}\underset{1}{\overset{e^{%
					\frac{1}{w}}}{\int }}...\underset{1}{\overset{e^{\frac{1}{w}}}{\int }}\left[ 
		\sum_{m=1}^{r}\left( -1\right) ^{1-m}\binom{r}{m}%
		f\left( e^{\frac{k}{w}}\left( t_{1}...t_{r}\right) ^{\frac{m}{r}}\right)
		-f\left( x\right) \right]\frac{dt_{1}}{t_{1}}...\frac{dt_{r}}{t_{r}}\right\} \\
		&=\left( \underset{\left\vert k-w\log x\right\vert \leq \frac{w\delta }{2}}{%
			\sum }+\underset{\left\vert k-w\log x\right\vert >\frac{w\delta }{2}}{\sum }
		\right) \chi \left( e^{-k}x^{w}\right) w^{r}\underset{1}{\overset{e^{\frac{1}{w}}}{\int }}...\underset{1}{%
			\overset{e^{\frac{1}{w}}}{\int }} \\
		&\times \left[ \sum_{m=1}^{r}\left( -1\right)
		^{1-m}\binom{r}{m}f\left( e^{\frac{k}{w}}\left( t_{1}...t_{r}\right) ^{\frac{%
				m}{r}}\right) -f\left( x\right) \right] \frac{dt_{1}}{t_{1}}...\frac{dt_{r}}{
			t_{r}}\\
		&=:E_{1}+E_{2}.
	\end{align*}
	Now we first estimate $E_{1}.$ Note that, if the integer $k$ is such that $%
	\left\vert k-w\log x\right\vert \leq \frac{w\delta }{2}$, for every $
	t_{j}\in \left[ 1,e^{\frac{1}{w}}\right] ,j=1,2,...,r$, we can write what follows:
	\begin{eqnarray*}
		\left\vert \log \left(e^{\frac{k}{w}}\left( t_{1}...t_{r}\right) ^{\frac{1}{r}%
		}\right)-\log x\right\vert 
		&\leq&\left\vert \frac{k}{w}-\log x\right\vert +\frac{1}{r}\left\vert \log
		\left( t_{1}...t_{r}\right) \right\vert \\
		&\leq &\frac{\delta }{2}+\frac{1}{w} \\
		&\leq &\delta
	\end{eqnarray*}
	for sufficiently large $w>0$ and moreover,
	\begin{eqnarray*}
		\left\vert \log \left(e^{\frac{k}{w}}\left( t_{1}...t_{r}\right) ^{\frac{m}{r}
		}\right)-\log \left(e^{\frac{k}{w}}\left( t_{1}...t_{r}\right) ^{\frac{m-1}{r}
		}\right)\right\vert&=&\left\vert \frac{m}{r}\log \left( t_{1}...t_{r}\right) -\frac{m-1}{r}\log
		\left( t_{1}...t_{r}\right) \right\vert \\
		&=&\left\vert \log \left( t_{1}...t_{r}\right) \left( \frac{m}{r}-\frac{m}{r}%
		+\frac{1}{r}\right) \right\vert \\
		&\leq &\frac{1}{w} \\
		&\leq &\delta
	\end{eqnarray*}
	
	for sufficiently large $w>0$, for every $m=2,3,...,r$. Thus, we get 
	\begin{align*}
		&\sum_{m=1}^{r}\left( -1\right) ^{1-m}\binom{r}{%
			m}f\left( e^{\frac{k}{w}}\left( t_{1}...t_{r}\right) ^{\frac{m}{r}}\right)
		-f\left( x\right) \\
		&=\sum_{m=1}^{r-1}\left( -1\right) ^{1-m}\left[ \binom{r-1}{m}+\binom{r-1}{%
			m-1}\right] f\left( e^{\frac{k}{w}}\left( t_{1}...t_{r}\right) ^{\frac{m}{r}%
		}\right) +\left( -1\right) ^{1-r}f\left( e^{\frac{k}{w}}\left(
		t_{1}...t_{r}\right) \right) -f\left( x\right)\\
		&=\sum_{m=1}^{r-2}\left( -1\right) ^{-m}\binom{r-1}{m}\left[ f\left( e^{%
			\frac{k}{w}}\left( t_{1}...t_{r}\right) ^{\frac{m+1}{r}}\right) -f\left( e^{%
			\frac{k}{w}}\left( t_{1}...t_{r}\right) ^{\frac{m}{r}}\right) \right] \\
		&+\left( -1\right) ^{1-r}\left[ f\left( e^{\frac{k}{w}}\left(
		t_{1}...t_{r}\right) \right) -f\left( e^{\frac{k}{w}}\left(
		t_{1}...t_{r}\right) ^{\frac{r-1}{r}}\right) \right]+\left[ f\left( e^{\frac{k}{w}}\left( t_{1}...t_{r}\right) ^{\frac{1}{r}%
		}\right) -f\left( x\right) \right].
	\end{align*}
Then we have 
	\begin{eqnarray*}
		\left\vert E_{1}\right\vert&\leq &\varepsilon \underset{\left\vert k-w\log x\right\vert \leq \frac{
				w\delta }{2}}{\sum }\left[ w^{r}\underset{1}{\overset{e^{\frac{1}{w}}}{\int }%
		}...\underset{1}{\overset{e^{\frac{1}{w}}}{\int }}\sum_{m=0}^{r-1}\binom{r-1%
		}{m}\frac{dt_{1}}{t_{1}}...\frac{dt_{r}}{t_{r}}\right] \left\vert \chi
		\left( e^{-k}x^{w}\right) \right\vert \\
		&\leq &2^{r-1}M_{0}\left( \chi \right) \varepsilon.
	\end{eqnarray*}
	On the other hand, we obtain 
	\begin{eqnarray*}
		\left\vert E_{2}\right\vert &\leq &\left( 2^{r}-1\right) 2\left\Vert f\right\Vert _{\infty }\underset{%
			\left\vert k-w\log x\right\vert >\frac{w\delta }{2}}{\sum }\left\vert \chi
		\left( e^{-k}x^{w}\right) \right\vert \\
		&<&\left( 2^{r+1}-2\right) \left\Vert f\right\Vert _{\infty }\varepsilon
	\end{eqnarray*}
	for sufficiently large $w>0$. This completes the first part of the proof. The second part of the theorem follows immediately by replacing the parameter $\delta >0$ used for the continuity of $f$ with the corresponding one for the uniform continuity of $f$, and by observing that all the aforementioned estimates hold uniformly for all $x \in \mathbb{R}_{+}$. This concludes the proof of the second part of the theorem.
\end{proof}

Now, we present $L^p,1\leq p<\infty$ convergence result for our operators. To do this, we begin with the following auxiliary results. First of all, from Theorem \ref{theorem1}, we have the following immediately.
\begin{theorem}\label{theorem2}
	Let $\chi\in\phi$ be a kernel and let $f\in CB_{\text{comp}}\left(\mathbb{R}_+\right)$ be fixed. Then:
		\begin{equation*}
		\lim_{w\rightarrow\infty}\left\Vert E_{w,r}^{\chi}f-f\right\Vert _{\infty }=0
	\end{equation*}
holds. 
\end{theorem}

\begin{theorem}\label{theorem3}
Let	$\chi\in\phi$ be a kernel such that $M_0\left(\chi\right)>0$. If $f\in CB_{\text{comp}}\left( \mathbb{R}_{+}\right) $, then we obtain
	\begin{equation*}
		\lim_{w\rightarrow\infty}\left\Vert E_{w,r}^{\chi}f-f\right\Vert _{p}=0,\text{ \ \ }1\leq p<\infty .
	\end{equation*}
\end{theorem}

\begin{proof}
	For this proof, we have to show that
	\begin{equation*}
		\lim_{w\rightarrow\infty}\int_{\mathbb{R}_+}
		\left\vert \left( E_{w,r}^{\chi}f\right) \left( x\right) -f\left( x\right)
		\right\vert ^{p}\frac{dx}{x}=0.
	\end{equation*}
	
	In order to do that, we make use of the Vitali convergence theorem (see, e.g. \cite{folland}).
	By Theorem \ref{theorem2}, it is necessary to demonstrate that the following conditions are satisfied:
	\begin{itemize}
		\item[(i)] For every $\varepsilon>0$, there exists $E_{\varepsilon
		}\in B\left( \mathbb{R}^{+}\right)$, where $B\left( \mathbb{R}_{+}\right) $ is the $\sigma -$algebra of all Lebesgue measurable subsets of $\mathbb{R}_{+}$, with $\mu \left( E_{\varepsilon }\right) <\infty $ and such that for
		every $F\in B\left( \mathbb{R}^{+}\right) $ with $F\cap E_{\varepsilon
		}=\emptyset$, we get
		\begin{equation*}
			\underset{F}{\int }\left\vert \left( E_{w,r}^{\chi}f\right) \left( x\right)
			\right\vert ^{p}\frac{dx}{x}<\varepsilon 
		\end{equation*}
	for sufficiently large $w>0$.
	\item[(ii)] For every $\varepsilon >0$, there exists $\delta >0$
	such that if $E\in B\left( \mathbb{R}^{+}\right) $ is such that $\underset{E}%
	{\int }\frac{dx}{x}<\delta $, we get
	\begin{equation*}
		\underset{E}{\int }\left\vert \left( E_{w,r}^{\chi}f\right) \left( x\right)
		\right\vert ^{p}\frac{dx}{x}<\varepsilon
	\end{equation*}
for sufficiently large $w>0$.
	\end{itemize}	

Regarding (i), assume that the support of $f$, denoted by $\operatorname{supp}f$, is contained within the interval $\left[ e^{-A},e^{A}\right]$ for some $A>0$. Now, let $\varepsilon >0$ and $\gamma >0$ such that $\gamma >A+1.$ Thus, we can infer that for every $k \notin \left[ -\gamma w, \gamma w \right]$ with $w \geq r$, the following integrals hold:
	\begin{equation}
		F_{r,e^{\frac{1}{w}}}\left( e^{\frac{k}{w}}\right) =w^{r}\underset{1}{%
			\overset{e^{\frac{1}{w}}}{\int }}...\underset{1}{\overset{e^{\frac{1}{w}}}{%
				\int }}\sum_{m=1}^{r}\left( -1\right) ^{1-m}%
		\binom{r}{m}f\left( e^{\frac{k}{w}}\left( t_{1}...t_{r}\right) ^{\frac{m}{r}%
		}\right) \frac{dt_{1}}{t_{1}}...\frac{dt_{r}}{t_{r}}=0. \label{NI}
	\end{equation}
	Indeed, if \ $k<-\gamma w$ and $t_{j}\in \left[ 1,e^{\frac{1}{w}}\right]
	,j=1,...,r$, we have 
	\begin{equation*}
		e^{\frac{k}{w}}\left( t_{1}...t_{r}\right) ^{\frac{m}{r}}\leq e^{-\gamma
		}e^{r.\frac{1}{wr}m}=e^{-\gamma }e^{\frac{m}{w}}<e^{-\gamma +1}<e^{-A}
	\end{equation*}
and moreover, if $k>\gamma w,$ we get
	\begin{equation*}
		e^{\frac{k}{w}}\left( t_{1}...t_{r}\right) ^{\frac{m}{r}}>e^{\gamma
			+1}>e^{A}
	\end{equation*}
from which we obtain that the integrals in \eqref{NI} are null. Therefore, via Lemma \ref{Lemma1}, with the above choice of $\gamma$ and for a fixed $\varepsilon > 0$, we know that there exists some $M > 0$ (which, without loss of generality, can be assumed to satisfy $M > A$) such that
	\begin{equation*}
		\underset{x\notin \left[ e^{-M},e^{M}\right] }{\int }w\left\vert \chi \left(
		e^{-k}x^{w}\right) \right\vert \frac{dx}{x}<\varepsilon
	\end{equation*}
for sufficiently large $w>0$ and $k\in \left[ -\gamma w,\gamma w\right] .$
Then, using Jensen inequality and Fubini-Tonelli theorem, we can write what follows:
	\begin{eqnarray*}
		\underset{x\notin \left[ e^{-M},e^{M}\right] }{\int }\left\vert \left(
		E_{w,r}^{\chi}f\right) \left( x\right) \right\vert ^{p}\frac{dx}{x} &=&
		\underset{x\notin \left[ e^{-M},e^{M}\right] }{\int }\left\vert \sum_{k\in\mathbb{Z}}\chi \left( e^{-k}x^{w}\right) F_{r,e^{\frac{1}{w}%
		}}\left( e^{\frac{k}{w}}\right) \right\vert ^{p}\frac{dx}{x} \\
		&\leq & \left[ M_{0}\left( \chi \right) \right]
		^{p-1}\underset{x\notin \left[ e^{-M},e^{M}\right] }{\int }\sum_{k\in\mathbb{Z}}\left\vert \chi \left(
		e^{-k}x^{w}\right) \right\vert\left\vert F_{r,e^{\frac{1}{w}}}\left( e^{\frac{k}{w}}\right)
		\right\vert ^{p}\frac{dx}{x} \\
		&=&\left[ M_{0}\left( \chi \right) \right] ^{p-1}\sum_{k\in\left[-\gamma w,\gamma w\right]}\left\vert F_{r,e^{\frac{1}{w}}}\left( e^{%
			\frac{k}{w}}\right) \right\vert ^{p}\underset{x\notin \left[ e^{-M},e^{M}\right] }{\int }\left\vert \chi \left( e^{-k}x^{w}\right) \right\vert \frac{dx}{x}.
	\end{eqnarray*}
Now, using twice Jensen inequality, we can easy to see that
	\begin{eqnarray}
		\left\vert F_{r,e^{\frac{1}{w}}}\left( e^{\frac{k}{w}}\right) \right\vert
		^{p} &=&\left\vert w^{r}\underset{1}{\overset{e^{\frac{1}{w}}}{\int }}...%
		\underset{1}{\overset{e^{\frac{1}{w}}}{\int }}\sum_{m=1}^{r}\binom{r}{m}f\left( e^{\frac{k}{w}}\left( t_{1}...t_{r}\right) ^{
			\frac{m}{r}}\right) \frac{dt_{1}}{t_{1}}...\frac{dt_{r}}{t_{r}}\right\vert
		^{p} \notag\\
		&\leq &w^{r}\underset{1}{\overset{e^{\frac{1}{w}}}{\int }}...\underset{1}{%
			\overset{e^{\frac{1}{w}}}{\int }}\left\vert \sum_{m=1}^{r}
		\binom{r}{m}f\left( e^{\frac{k}{w}}\left( t_{1}...t_{r}\right) ^{\frac{m}{r}%
		}\right) \right\vert ^{p}\frac{dt_{1}}{t_{1}}...\frac{dt_{r}}{t_{r}} \notag\\
		&\leq &\left( 2^{r}-1\right) ^{p-1}w^{r}\underset{1}{\overset{e^{\frac{1}{w}}%
			}{\int }}...\underset{1}{\overset{e^{\frac{1}{w}}}{\int }}\sum_{m=1}^{r}\binom{r}{m}\left\vert f\left( e^{\frac{k}{w}}\left(
		t_{1}...t_{r}\right) ^{\frac{m}{r}}\right) \right\vert ^{p}\frac{dt_{1}}{%
			t_{1}}...\frac{dt_{r}}{t_{r}}. \label{ineq3.4}
	\end{eqnarray}
Then, we obtain
	\begin{align*}
		&\underset{x\notin \left[ e^{-M},e^{M}\right] }{\int }\left\vert \left(
		E_{w}^{\chi ,r}f\right) \left( x\right) \right\vert ^{p}\frac{dx}{x}\\
		 &\leq
		\left( 2^{r}-1\right) ^{p-1}\left[ M_{0}\left( \chi \right) \right]
		^{p-1}w^{r-1}\sum_{\left|k\right|\leq\gamma w}\underset
		{1}{\overset{e^{\frac{1}{w}}}{\int }}...\underset{1}{\overset{e^{\frac{1}{w}}%
			}{\int }}\sum_{m=1}^{r}\binom{r}{m}\left\vert f\left( e^{%
			\frac{k}{w}}\left( t_{1}...t_{r}\right) ^{\frac{m}{r}}\right) \right\vert
		^{p}\frac{dt_{1}}{t_{1}}...\frac{dt_{r}}{t_{r}}\underset{\left\vert
			x\right\vert >e^{M}}{\int }\left\vert \chi \left( e^{-k}x^{w}\right)
		\right\vert \frac{dx}{x} \\
		&<\varepsilon \left( 2^{r}-1\right) ^{p-1}\left[ M_{0}\left( \chi \right) %
		\right] ^{p-1}\underset{\left\vert k\right\vert \leq \gamma w}{\sum }w^{r-1}%
		\underset{1}{\overset{e^{\frac{1}{w}}}{\int }}...\underset{1}{\overset{e^{
		\frac{1}{w}}}{\int }}\sum_{m=1}^{r}\binom{r}{m} \left\vert f\left( e^{\frac{k}{w}}\left( t_{1}...t_{r}\right) ^{
		\frac{m}{r}}\right) \right\vert ^{p}\frac{dt_{1}}{t_{1}}...\frac{dt_{r}}{t_{r}} \\
		&\leq \varepsilon \left( 2^{r}-1\right) ^{p}\left[ M_{0}\left( \chi \right) %
		\right] ^{p-1}\left\Vert f\right\Vert _{\infty }^{p}w^{-1}\left[ \underset{%
			\left\vert k\right\vert \leq \gamma w}{\sum }1\right] \\
		&\leq \varepsilon \left( 2^{r}-1\right) ^{p}\left[ M_{0}\left( \chi \right) %
		\right] ^{p-1}\left\Vert f\right\Vert _{\infty }^{p}w^{-1}\left( 2\gamma
		+1\right)
	\end{align*}
for sufficiently large $w>0$. From this point, it is clear that (i)  immediately follows by setting $E_{\varepsilon}:=\left[e^{-M},e^M\right]$ and $F:=\mathbb{R}_+\backslash E_{\varepsilon}$. 

As for (ii), let $E\subset\mathbb{R}_+$ be a measurable set such that $\int_{E}\frac{dx}{x}<\delta$.  Since $f\in CB_{\text{comp}}\left(\mathbb{R}_+\right)$, then by uisng the inequality \eqref{ineq3.2}, we immediately have that
\begin{equation*}
	\left|\left(E_{w,r}^\chi f\right)\left(x\right)\right|\leq\left(2^r-1\right)^p\left\|f \right\|_\infty^p\left[M_0\left(\chi\right)\right]^p. 
\end{equation*}
Thus, we obtain
\begin{eqnarray*}
	\int_{E}\left|\left(E_{w,r}^\chi f\right)\left(x\right)\right|^p\frac{dx}{x}&\leq&\mu\left(E\right)\left(2^r-1\right)^p\left\|f \right\|_\infty^p\left[M_0\left(\chi\right)\right]^p\\
	&<&\delta\left(2^r-1\right)^p\left\|f \right\|_\infty^p\left[M_0\left(\chi\right)\right]^p.
\end{eqnarray*}
Now, the assertion follows by taking
\begin{equation*}
	\delta:=\frac{\varepsilon}{\left(2^r-1\right)^p\left\|f \right\|_\infty^p\left[M_0\left(\chi\right)\right]^p}. 
\end{equation*}
This completes the proof of theorem. 
\end{proof}
The following inequality demonsrates that the Mellin-Steklov type exponential sampling operators are well-defined for functions in the space $L^{p}$.
\begin{theorem}\label{theorem4}
Let $\chi\in\phi$ be a kernel. If $f\in L^{p}\left( \mathbb{R}_{+}\right) ,$ $%
	1\leq p<\infty$, then the inequality
	\begin{equation*}
		\left\Vert E_{w,r}^{\chi}f\right\Vert _{p}\leq \left( 2^{r}-1\right) \left[
		M_{0}\left( \chi \right) \right] ^{\frac{p-1}{p}}\left\Vert \chi \right\Vert
		_{1}^{\frac{1}{p}}\left\Vert f\right\Vert _{p}
	\end{equation*}
holds.
\end{theorem}

\begin{proof}
As in the proof of the previous theorem, using Jensen inequality and the Fubini-Tonelli theorem, we can write what follows:
	\begin{equation*}
		\left\Vert E_{w,r}^{\chi}f\right\Vert _{p}^{p}
		\leq\left[ M_{0}\left( \chi \right) \right] ^{p-1}\sum_{k\in\mathbb{Z}}\left\vert F_{r,e^{\frac{1}{w}}}\left( e^{\frac{k}{w}}\right)
		\right\vert ^{p}\underset{\mathbb{R}^{+}}{\int }\left\vert \chi \left(
		e^{-k}x^{w}\right) \right\vert \frac{dx}{x}.
	\end{equation*}
	Now, by applying the change of variable $x=t^{\frac{1}{w}}e^{\frac{k}{w}},$ we arrive at the following:
	\begin{equation*}
		\left\Vert E_{w,r}^{\chi}f\right\Vert _{p}^{p}\leq \left[ M_{0}\left( \chi
		\right) \right] ^{p-1}\left\Vert \chi \right\Vert _{1}w^{-1}\sum_{k\in\mathbb{Z}}\left\vert F_{r,e^{\frac{1}{w}}}\left( e^{\frac{k}{w}%
		}\right) \right\vert ^{p}.
	\end{equation*}
Using the inequality \eqref{ineq3.4} and applying the change of variable $y=e^{\frac{k}{w}}\left( t_{1}\right) ^{\frac{m}{r}},m=1,...,r,$ we can write that 
	\begin{eqnarray*}
		\left\vert F_{r,e^{\frac{1}{w}}}\left( e^{\frac{k}{w}}\right) \right\vert
		^{p}
		&\leq &\left( 2^{r}-1\right) ^{p-1}w^{r}\underset{1}{\overset{e^{\frac{1}{w}}%
			}{\int }}...\underset{1}{\overset{e^{\frac{1}{w}}}{\int }}\sum_{m=1}^{r}\binom{r}{m}\left\vert f\left( e^{\frac{k}{w}}\left(
		t_{1}...t_{r}\right) ^{\frac{m}{r}}\right) \right\vert ^{p}\frac{dt_{1}}{%
			t_{1}}...\frac{dt_{r}}{t_{r}} \\
		&= &\left( 2^{r}-1\right) ^{p-1}w^{r}\sum_{m=1}^{r}
		\binom{r}{m}\underset{1}{\overset{e^{\frac{1}{w}}}{\int }}...\underset{1}{
		\overset{e^{\frac{1}{w}}}{\int }}\left\{ \underset{e^{\frac{k}{w}}}{\overset{e^{\frac{k}{w}+\frac{m}{rw}}}{\int }}\left\vert f\left( y\left(
		t_{2}...t_{r}\right) ^{\frac{m}{r}}\right) \right\vert ^{p}\frac{dy}{y}%
		\right\} \frac{dt_{2}}{t_{2}}...\frac{dt_{r}}{t_{r}} \\
		&\leq &\left( 2^{r}-1\right) ^{p-1}w^{r}\sum_{m=1}^{r}
		\binom{r}{m}\underset{1}{\overset{e^{\frac{1}{w}}}{\int }}...\underset{1}{%
			\overset{e^{\frac{1}{w}}}{\int }}\left\{ \underset{e^{\frac{k}{w}}}{\overset{%
				e^{\frac{k+1}{w}}}{\int }}\left\vert f\left( y\left( t_{2}...t_{r}\right) ^{%
			\frac{m}{r}}\right) \right\vert ^{p}\frac{dy}{y}\right\} \frac{dt_{2}}{t_{2}}%
		...\frac{dt_{r}}{t_{r}}.
	\end{eqnarray*}
Thus, we have 
	\begin{align*}
		&\left\Vert E_{w,r}^{\chi}f\right\Vert _{p}^{p}\\
		 &\leq \left[ M_{0}\left(
		\chi \right) \right] ^{p-1}\left\Vert \chi \right\Vert _{1}\left(
		2^{r}-1\right) ^{p-1}w^{r-1}\sum_{m=1}^{r}\binom{r}{m}%
		\underset{1}{\overset{e^{\frac{1}{w}}}{\int }}...\underset{1}{\overset{e^{%
					\frac{1}{w}}}{\int }}\left\{ \sum_{k\in\mathbb{Z}}\underset{e^{\frac{k}{w}}}{%
			\overset{e^{\frac{k+1}{w}}}{\int }}\left\vert f\left( y\left(
		t_{2}...t_{r}\right) ^{\frac{m}{r}}\right) \right\vert ^{p}\frac{dy}{y}%
		\right\} \frac{dt_{2}}{t_{2}}...\frac{dt_{r}}{t_{r}} \\
		&=\left[ M_{0}\left( \chi \right) \right] ^{p-1}\left\Vert \chi \right\Vert
		_{1}\left( 2^{r}-1\right) ^{p-1}w^{r-1}\sum_{m=1}^{r}%
		\binom{r}{m}\underset{1}{\overset{e^{\frac{1}{w}}}{\int }}...\underset{1}{%
			\overset{e^{\frac{1}{w}}}{\int }}  \frac{dt_{2}}{t_{2}}...\frac{dt_{r}}{t_{r}}\left\Vert f\left( \cdot
		\right) \left( t_{2}...t_{r}\right) ^{\frac{m}{r}} \right\Vert
		_{p}^{p}.
	\end{align*}
	Now, recalling that
	\begin{equation*}
		\left\Vert f\left( \cdot \right)  \left( t_{2}...t_{r}\right) ^{\frac{m
			}{r}} \right\Vert _{p}^{p}=\left\Vert f\left( \cdot \right)
		\right\Vert _{p}^{p},\text{ }
	\end{equation*}
for every $t_{j}\in \left[ 1,e^{\frac{1}{w}}\right] ,$ $j=2,3,...,r; m=1,...,r,$ we finally have:
	\begin{eqnarray*}
		\left\Vert E_{w,r}^{\chi}f\right\Vert _{p}^{p} &\leq &\left[ M_{0}\left(
		\chi \right) \right] ^{p-1}\left\Vert \chi \right\Vert _{1}\left(
		2^{r}-1\right) ^{p-1}\left\Vert f\right\Vert _{p}^{p}w^{r-1}\sum_{m=1}^{r}\binom{r}{m}\underset{1}{\overset{e^{\frac{1}{w}}}{%
				\int }}...\underset{1}{\overset{e^{\frac{1}{w}}}{\int }}\frac{dt_{2}}{t_{2}}%
		...\frac{dt_{r}}{t_{r}} \\
		&\leq &\left[ M_{0}\left( \chi \right) \right] ^{p-1}\left\Vert \chi
		\right\Vert _{1}\left( 2^{r}-1\right) ^{p}\left\Vert f\right\Vert _{p}^{p}. 
	\end{eqnarray*}
Hence, the proof is completed. 
\end{proof}

Now, we are able to $L^p$ convergence result for our operators. 

\begin{theorem}
	Let $\chi\in\phi$ be a kernel and let $f\in L^{p}\left( \mathbb{R}_{+}\right) ,1\leq p<\infty,$ be a function.  Then
	\begin{equation*}
		\lim_{w\rightarrow\infty}\left\Vert E_{w,r}^{\chi}f-f\right\Vert_p =0
	\end{equation*}
holds.
\end{theorem}
\begin{proof}
Since $CB_{\text{comp}}\left(\mathbb{R}_{+}\right) $ is dense in $L^{p}\left( \mathbb{R}_{+}\right)$ (see, \cite{barmansch}), for every fixed $\ \varepsilon >0$, there exists a $g\in CB_{\text{comp}}\left( \mathbb{R}_{+}\right) $ such that $\left\Vert g-f\right\Vert_{p}<\varepsilon.$  Hence, by Theorem \ref{theorem4}, we get
	\begin{eqnarray*}
		\left\Vert E_{w,r}^{\chi}f-f\right\Vert _{p} &=&\left\Vert E_{w,r}^{\chi}f-E_{w,r}^{\chi}g+E_{w,r}^{\chi}g-g+g-f\right\Vert _p\\
		&\leq &\left\Vert E_{w,r}^{\chi}f-E_{w,r}^{\chi }g\right\Vert
		_{p}+\left\Vert E_{w,r}^{\chi }g-g\right\Vert _{p}+\left\Vert g-f\right\Vert_{p} \\
		&=&\left\Vert E_{w,r}^{\chi}\left( f-g\right) \right\Vert _{p}+\left\Vert
		E_{w,r}^{\chi}g-g\right\Vert _{p}+\left\Vert g-f\right\Vert _{p} \\
		&\leq &\left\Vert g-f\right\Vert _{p}\left\{ \left[ M_{0}\left( \chi \right) %
		\right] ^{\frac{p-1}{p}}\left\Vert \chi \right\Vert _{1}^{\frac{1}{p}}\left(
		2^{r}-1\right) \left\Vert f\right\Vert _{p}\right\} +\left\Vert E_{w,r}^{\chi}g-g\right\Vert _{p} \\
		&\leq &\left[ M_{0}\left( \chi \right) %
		\right] ^{\frac{p-1}{p}}\left\Vert \chi \right\Vert _{1}^{\frac{1}{p}}\left(
		2^{r}-1\right) \left\Vert f\right\Vert _{p}\varepsilon +\left\Vert E_{w,r}^{\chi}g-g\right\Vert _{p}.
	\end{eqnarray*}
Passing to limit for $w\rightarrow \infty $ by using Theorem \ref{theorem3}, the proof is completed. 
\end{proof}

\end{document}